\documentclass[a4paper]{article}
\usepackage{fullpage}
\usepackage{amsmath}
\usepackage{amssymb}
\usepackage{amsthm}
\usepackage{graphicx}
\usepackage{bbm}
\usepackage{enumerate}
\usepackage{microtype}
\usepackage{xcolor}
\usepackage[none]{hyphenat}
\usepackage[colorlinks=true,urlcolor=blue,linkcolor=blue,plainpages=false,pdfpagelabels]{hyperref}
\allowdisplaybreaks

\newtheorem{theorem}{Theorem}[section]
\newtheorem{proposition}[theorem]{Proposition}
\newtheorem{lemma}[theorem]{Lemma}

\newtheorem{definition}[theorem]{Definition}

\newcommand{\N}{\mathbb{N}}

\newcommand{\eps}{\varepsilon}

\title{Construction of fixed points of asymptotically nonexpansive mappings in uniformly convex hyperbolic spaces}
\author{Andrei Sipo\c s${}^{a,b}$\\[2mm]
\footnotesize ${}^a$Research Center for Logic, Optimization and Security (LOS), Department of Computer Science,\\
\footnotesize Faculty of Mathematics and Computer Science, University of Bucharest,\\
\footnotesize Academiei 14, 010014 Bucharest, Romania\\[1mm]
\footnotesize ${}^b$Simion Stoilow Institute of Mathematics of the Romanian Academy,\\
\footnotesize Calea Grivi\c tei 21, 010702 Bucharest, Romania\\[2mm]
\footnotesize E-mail: andrei.sipos@fmi.unibuc.ro\\
}
\date{}

\begin{document}

\maketitle

\begin{abstract}
Kohlenbach and Leu\c stean have shown in 2010 that any asymptotically nonexpansive self-mapping of a bounded nonempty $UCW$-hyperbolic space has a fixed point. In this paper, we adapt a construction due to Moloney in order to provide a sequence that converges strongly to such a fixed point.


\end{abstract}

\section{Introduction}

In 1972, Goebel and Kirk generalized \cite{GoeKir72} the classical Browder-G\"ohde-Kirk theorem to the class of {\it asymptotically nonexpansive mappings} (also introduced in that paper), which are mappings $T$ having the property that there is a $(k_n) \subseteq [0,\infty)$ with $\lim_{n \to \infty} k_n = 0$ such that for any $x$, $y$ in the domain of $T$ and any $n \in \N$,
$$d(T^nx,T^ny) \leq (1+k_n)d(x,y),$$ 
i.e. they showed that any self-mapping of a bounded closed convex nonempty subset of a uniformly convex Banach space with the above property has a fixed point.

In the recent decades, there has been a renewed interest in fixed point theory and convex optimization as practiced in nonlinear generalizations of the classical structures of functional analysis. For example, there exists in the literature a number of definitions of a notion of a `hyperbolic space' \cite{BH99,GoeKir83,GoeRei84,Kir82,ReiSha90,Tak70}, that aim to axiomatize the convexity structure of normed spaces. The kind of spaces that we shall employ here have a particularly flexible definition due to Kohlenbach \cite{Koh05} and are called {\it $W$-hyperbolic spaces} (see the next section for a definition and \cite[pp. 384--387]{Koh08} for a detailed discussion on the relationship between various definitions of hyperbolicity).

Uniform convexity, a property originally due to Clarkson \cite{Cla36}, was generalized to this hyperbolic setting (following \cite[p. 105]{GoeRei84}) by Leu\c stean in \cite{Leu07}. The subclass that is the most natural generalization of uniformly convex Banach spaces has then been identified with the one having a {\it monotone} modulus of uniform convexity. Those spaces have been called {\it $UCW$-hyperbolic spaces} in \cite{Leu10} (see also \cite{Leu14}), where  the corresponding Browder-G\"ohde-Kirk result was proven for spaces of this kind which are complete and nonempty. The Goebel-Kirk extension mentioned above for asymptotically nonexpansive mappings was obtained in the same setting by Kohlenbach and Leu\c stean in \cite{KohLeu10}.

Thirty years ago, Moloney showed \cite{Mol89,Mol94} (by refining an earlier result of Kaniel \cite{Kan71}) how to explicitly construct, for any asymptotically nonexpansive self-mapping of a bounded closed convex nonempty subset of a uniformly convex Banach space, a sequence that converges strongly to one of its fixed points. What we do in this paper is to show that this construction may be adapted to also work in a bounded nonempty $UCW$-hyperbolic space.

\section{Facts on hyperbolic spaces}

As stated in the Introduction, the following definition is due to Kohlenbach \cite{Koh05}.

\begin{definition}
A {\bf $W$-hyperbolic space} is a triple $(X,d,W)$ where $(X,d)$ is a metric space and $W: X^2 \times [0,1] \to X$ such that for all $x$, $y$, $z$, $w \in X$ and $\lambda$, $\mu \in [0,1]$, we have that
\begin{enumerate}[(W1)]
\item $d(z,W(x,y,\lambda)) \leq (1-\lambda)d(z,x) + \lambda d(z,y)$;
\item $d(W(x,y,\lambda),W(x,y,\mu) = |\lambda-\mu|d(x,y)$;
\item $W(x,y,\lambda)=W(y,x,1-\lambda)$;
\item $d(W(x,z,\lambda),W(y,w,\lambda))\leq(1-\lambda) d(x,y) + \lambda d(z,w)$.
\end{enumerate}
\end{definition}

Clearly, any normed space may be made into a $W$-hyperbolic space in a canonical way. As per \cite{Koh05,Leu07}, a particular nonlinear class of $W$-hyperbolic spaces is the one of CAT(0) spaces, introduced by A. Aleksandrov \cite{Ale51} and named as such by M. Gromov \cite{Gro87}.

A subset $C$ of a $W$-hyperbolic space $(X,d,W)$ is called {\it convex} if for any $x$, $y \in C$ and $\lambda \in [0,1]$, $W(x,y,\lambda) \in C$. If $(X,d,W$) is a $W$-hyperbolic space, $x$, $y \in X$ and $\lambda \in [0,1]$, we denote the point $W(x,y,\lambda)$ by $(1-\lambda)x+\lambda y$. We will mainly write $\frac{x+y}2$ for $\frac12 x + \frac12 y$. The following properties are immediate consequences of the definition of a $W$-hyperbolic space.

\begin{proposition}
Let $(X,d,W)$ be a $W$-hyperbolic space. Let $x$, $y \in X$ and $\lambda \in [0,1]$. Then we have:
\begin{enumerate}[(i)]
\item $1x + 0y = x$;
\item $0x + 1y = y$;
\item $(1-\lambda) x + \lambda x =x$;
\item $d(x,(1-\lambda)x+\lambda y) = \lambda d(x,y)$;
\item $d(y,(1-\lambda)x+\lambda y) = (1-\lambda) d(x,y)$.
\end{enumerate}
\end{proposition}

\begin{definition}
If $(X,d,W)$ is a $W$-hyperbolic space, then a {\bf modulus of uniform convexity} for $(X,d,W)$ is a function $\eta :(0, \infty) \times (0,\infty) \to (0,1]$ such that for any $r$, $\eps >0$ and any $a$, $x$, $y \in X$ with $d(x,a) \leq r$, $d(y,a) \leq r$, $d(x,y) \geq \eps r$ we have that
$$d\left(\frac{x+y}2,a\right) \leq (1-\eta(r,\eps))r.$$
We call the modulus {\bf monotone} if for any $r$, $s$, $\eps>0$ with $s \leq r$, we have $\eta(r,\eps) \leq \eta(s,\eps)$.
\end{definition}

\begin{definition}
A {\bf $UCW$-hyperbolic space} is a quadruple $(X,d,W,\eta)$ where $(X,d,W)$ is a $W$-hyperbolic space and $\eta$ is a monotone modulus of uniform convexity for $(X,d,W)$.
\end{definition}

As remarked in \cite[Proposition 2.6]{Leu07}, CAT(0) spaces are $UCW$-hyperbolic spaces having as a modulus of uniform convexity the function $(r,\eps) \mapsto \frac{\eps^2}8$, quadratic in $\eps$. Note that a closed convex nonempty subset of a (complete) $(UC)W$-hyperbolic space is itself a (complete) nonempty $(UC)W$-hyperbolic space (in contrast to e.g. normed spaces).

The following is an adaptation of a result due to Kohlenbach and Leu\c stean, namely \cite[Lemma 3.2]{KohLeu09}.

\begin{proposition}\label{u-prop}
Let $(X,d,W,\eta)$ be a $UCW$-hyperbolic space. Define, for any $r$, $\eps > 0$, $u(r,\eps):= \frac{\eps}2 \cdot \eta(r,\eps)$. Then, for any $r$, $\eps>0$ and any $a$, $x$, $y \in X$ with $d(x,a) \leq d(y,a) \leq r$ and $d(x,y) \geq \eps r$ we have that
$$d\left(\frac{x+y}2,a\right) \leq d(y,a)- u(r,\eps) r.$$
In addition, if there is a function $\eta'$ which is nondecreasing in its second argument such that for all $r$ and $\eps$, $\eta(r,\eps)=\eps\eta'(r,\eps)$ (e.g. in the case of CAT(0) spaces, as per the above remark), then one can take $u$ to be simply $\eta$.
\end{proposition}

\begin{proof}
Let $r$, $\eps>0$ and $a$, $x$, $y \in X$ be as required. First, note that
$$\frac{\eps r}2\leq\frac{d(x,y)}2 \leq\frac{d(x,a)+d(y,a)}2 \leq d(y,a),$$
so, using that $\eta$ is a monotone modulus of uniform convexity, we get that
$$d\left(\frac{x+y}2,a\right) \leq (1-\eta(d(y,a),\eps)) \cdot d(y,a) \leq (1-\eta(r,\eps)) \cdot d(y,a) \leq d(y,a) - \eta(r,a)\frac{\eps r}2 = d(y,a)- u(r,\eps)r.$$
The hypotheses imply that $d(y,a)\neq 0$, so we may write
$$d(x,y) \geq \eps r = \frac{\eps r}{d(y,a)} \cdot d(y,a),$$
and therefore, in the second case,
\begin{align*}
d\left(\frac{x+y}2,a\right) &\leq \left(1-\eta\left(d(y,a),\frac{\eps r}{d(y,a)}\right)\right) \cdot d(y,a) = d(y,a) - d(y,a) \cdot \frac{\eps r}{d(y,a)} \cdot \eta'\left(d(y,a),\frac{\eps r}{d(y,a)}\right)\\
&=d(y,a) - \eps r \eta'\left(d(y,a),\frac{\eps r}{d(y,a)}\right) \leq d(y,a) - \eps r \eta'(d(y,a),\eps) \\
&= d(y,a) - r \eta(d(y,a),\eps) \leq d(y,a) -  \eta(r,\eps)r.
\end{align*}
\end{proof}

\begin{definition}
Let $(X,d)$ be a metric space, $T:X\to X$ and $(k_n) \subseteq [0,\infty)$ such that $\lim_{n \to \infty} k_n = 0$. Then $T$ is called {\bf asymptotically nonexpansive} with respect to $(k_n)$ if for any $x$, $y \in X$ and any $n \in \N$,
$$d(T^nx,T^ny) \leq (1+k_n)d(x,y).$$
\end{definition}

For any self-mapping $T$ (of an arbitrary set), we denote the set of its fixed points by $Fix(T)$. In \cite{KohLeu10}, Kohlenbach and Leu\c stean have proved that any asymptotically nonexpansive self-mapping of a bounded complete nonempty $UCW$-hyperbolic space has a fixed point.

\section{Main results}

We fix a complete nonempty $UCW$-hyperbolic space $(X,d,W,\eta)$ and $b>0$ an upper bound for its diameter. Let $(k_n) \subseteq [0,\infty)$ be such that $\lim_{n \to \infty} k_n = 0$ and $T: X \to X$ be asymptotically nonexpansive with respect to $(k_n)$, so $Fix(T) \neq \emptyset$.

We shall construct a mapping $S: X \to X$ such that:
\begin{enumerate}[(i)]
\item $Fix(T)=Fix(S)$ (so $Fix(S) \neq \emptyset$);
\item for any $p \in Fix(S)$ and $x\in X$, $d(Sx,p)\leq d(x,p)$;
\item for any $(x_n) \subseteq X$ having $x \in X$ as its limit and with $\lim_{n\to\infty} d(x_n,Sx_n)=0$, we have $x \in Fix(S)$.
\end{enumerate}

Note that if $T$ is nonexpansive we may simply take $S:=T$.

\begin{lemma}\label{lm}
Let $x \in X$. Then for any $n \in \N$ there is an $m \in \{n,n+1\}$ such that
$$d(T^mx,x) \geq \frac1{2+k_1} d(Tx,x).$$
\end{lemma}

\begin{proof}
Let $n \in \N$. Assume by way of contradiction that
$$d(T^nx,x) < \frac1{2+k_1} d(Tx,x)$$
and
$$d(T^{n+1}x,x) < \frac1{2+k_1} d(Tx,x).$$
Then, since
$$d(T^{n+1}x,Tx) \leq (1+k_1)d(T^nx,x) < \frac{1+k_1}{2+k_1}d(Tx,x),$$
we have that
$$d(Tx,x) \leq d(T^{n+1}x,x) + d(T^{n+1}x,Tx)  < \frac1{2+k_1} d(Tx,x) + \frac{1+k_1}{2+k_1}d(Tx,x) = d(Tx,x),$$
a contradiction.
\end{proof}

Let now $x \in X$. If $Tx=x$, put $Sx:=x$. If $Tx \neq x$, then put $n$ be minimal such that $k_n$ and $k_{n+1}$ are both smaller or equal than
$$\min\left(\frac{d(x,Tx)}{2b}, 2\eta\left(b,\frac{d(x,Tx)}{2b}\right)\right) > 0$$
(such an $n$ exists since $\lim_{n \to \infty} k_n = 0$). Then, making use of Lemma~\ref{lm}, put $m \in \{n,n+1\}$ be minimal such that
$$d(T^mx,x) \geq \frac1{2+k_1} d(Tx,x)$$
and set 
$$Sx:=\frac{T^mx + x}2.$$

The following proposition shows that $S$ has all the required properties.

\begin{proposition}
Let $x \in X$, $p \in Fix(S)$ and $(x_n) \subseteq X$. We have that:
\begin{enumerate}[(i)]
\item $d(Sx,x) \geq \frac1{2(2+k_1)} d(Tx,x)$;\label{p1}
\item $Fix(T)=Fix(S)$;
\item $d(Sx,p)\leq d(x,p)$;
\item if $x$ is the limit of $(x_n)$ and $\lim_{n\to\infty} d(x_n,Sx_n)=0$, we have $x \in Fix(S)$.
\end{enumerate}
\end{proposition}

\begin{proof}
\begin{enumerate}[(i)]
\item If $Tx=x$, there is nothing to show. If $Tx \neq x$, then, by putting $m$ to be the one from the construction of $Sx$, we have that $Sx=\frac{T^mx + x}2$, so $d(Sx,x) = \frac12 d(T^mx,x) \geq \frac1{2(2+k_1)} d(Tx,x)$.
\item The inclusion $Fix(T)\subseteq Fix(S)$ follows by the construction of $S$; $Fix(S)\subseteq Fix(T)$ follows by (\ref{p1}).
\item We have that $p \in Fix(T)$. If $Tx=x$, there is nothing to show. Suppose, then, that $Tx \neq x$, so, again by putting $m$ to be the one from the construction of $Sx$, we have that $Sx=\frac{T^mx + x}2$.

Put $c:=\frac1{(2+k_1)}d(Tx,x)$, $q:=d(x,p)$ and $\eps:=\frac{c}{2b}$. Since $m$ was chosen such that $k_m \leq \eps$, we have that
$$b(k_m + \eps) \leq 2b\eps = c.$$
Therefore, since $q \leq b$,
$$d(T^mx,x) \geq c \geq b(k_m+\eps) \geq q(k_m +\eps),$$
so $d(T^mx,x) - qk_m \geq q\eps$. In addition, using the monotonicity of $\eta$ and (again) the way $m$ was chosen,
$$\eta(q,\eps) \geq \eta(b,\eps) \geq k_m/2.$$

If $d(T^mx,p)<q$, then
$$d(Sx,p) = d\left(\frac{T^mx + x}2,p\right) \leq \frac12 d(T^mx,p) + \frac12 d(x,p) \leq q.$$
Consider now the case where $d(T^mx,p) \geq q$. Then $d(T^mx,p) \leq (1+k_m) d(x,p) = (1+k_m)q$, so there is a $\beta \in [0,k_m]$ with $d(T^mx,p)=(1+\beta)q$. Put
$$y:=\frac1{1+\beta}T^m x+ \frac\beta{1+\beta} p.$$
Then
$$d(T^mx, y) = \frac\beta{1+\beta} d(T^mx,p) = q\beta \leq qk_m,$$
so
$$d(T^mx,x) \leq d(x,y) + d(T^mx,y) = d(x,y) + qk_m$$
and thus
$$d(x,y) \geq d(T^mx,x) - qk_m \geq q\eps.$$
On the other hand, we have that $d(x,p)=q$ and
$$d(y,p) = \frac1{1+\beta} d(T^mx,p) = q,$$
so, since $\eta$ is a modulus of uniform convexity,
$$d\left(\frac{x+y}2,p\right) \leq (1-\eta(q,\eps))q.$$
Then
\begin{align*}
d(Sx,p) &= d\left(\frac{T^mx + x}2,p\right) \leq d\left(\frac{x+y}2,p\right) + d\left(\frac{x+y}2,\frac{x+T^mx}2\right)\\
&\leq (1-\eta(q,\eps))q + \frac12 d(T^mx,y) \leq \left(1-\frac{k_m}2\right)q + \frac{qk_m}2 = q.
\end{align*}

\item From (\ref{p1}), we get that $\lim_{n\to\infty} d(x_n,Tx_n)=0$, then, by the continuity of $T$, we get $x \in Fix(T) = Fix(S)$.
\end{enumerate}
\end{proof}

Given $S$ with these properties and $x \in X$, we will now construct a sequence converging to a fixed point of $S$.

For any $j \geq 1$, we shall set an $m_j \geq 1$ and a finite sequence $(z_{ij})_{i=1}^{m_j}$, and we shall put $y_j:=z_{m_jj}$. The sequence $(y_j)$ will be the one we are after.

Put $m_1:=1$ and $z_{11}:=x$. Assume now that we have constructed the $j$th finite sequence and we are seeking the next one. We distinguish two cases.\\[2mm]

{\bf Construction case I.} There is an $i \in [2,m_j-1]$ such that $d(z_{ij},z_{(i+1)j}) < d(z_{ij},z_{(i-1)j})$ and $d\left(z_{ij},\frac{z_{ij}+Sz_{m_jj}}2\right) \geq d(z_{ij},z_{(i-1)j})$.\\[1mm]

Let $i$ be minimal with this property. Then put $m_{j+1}:=i+1$, put for all $k \leq i$, $z_{k(j+1)}:=z_{kj}$ and $y_{j+1}=z_{(i+1)(j+1)}:=\frac{z_{ij}+Sz_{m_jj}}2$.\\[2mm]

{\bf Construction case II.} There is no such $i$.\\[1mm]

In this case, put $m_{j+1}:=m_j+1$, put for all $k \leq m_j$, $z_{k(j+1)}:=z_{kj}$ and $z_{(m_j+1)(j+1)}:=\frac{z_{m_jj}+Sz_{m_jj}}2$.\\[2mm]

It is immediate that $m_2=2$ and $m_3=3$. By a simple induction, it follows that for all $j \geq 3$, $m_j \geq 3$.

\begin{lemma}\label{l4}
Let $p \in Fix(S)$, $j \in \N$ and $i \in [1,m_j-1]$. Then $d(z_{(i+1)j},p) \leq d(z_{ij},p)$.
\end{lemma}

\begin{proof}
We prove this by induction on $j$. If $j=1$, the property holds vacuously. Suppose now that the property holds for $j$ and we want to prove it for $j+1$. If $i<m_{j+1}-1$, $z_{(i+1)(j+1)}=z_{(i+1)j}$ and $z_{i(j+1)}=z_{ij}$, so we simply apply the induction hypothesis. If $i=m_{j+1}-1$, then $z_{i(j+1)}=z_{ij}$ and $z_{(i+1)(j+1)} = \frac{z_{ij}+Sz_{m_jj}}2$, so
\begin{align*}
d(z_{(i+1)(j+1)},p) &= d\left(\frac{z_{ij}+Sz_{m_jj}}2,p\right) \leq \frac12d(z_{ij},p) + \frac12d(Sz_{m_jj},p) \leq \frac12d(z_{ij},p) + \frac12d(z_{m_jj},p) \\
&\leq \frac12d(z_{ij},p) + \frac12d(z_{ij},p) = d(z_{ij},p) = d(z_{i(j+1)},p).
\end{align*}
\end{proof}

\begin{lemma}\label{l3}
Let $p \in Fix(S)$, $\eps>0$, $j \in \N$ and $i \in [1,m_j-1]$. Let $u$ be such that the property described by Proposition~\ref{u-prop} holds. Assume that $d(z_{ij},z_{(i+1)j}) \geq \eps$. Then $ d(z_{ij},p) - d(z_{(i+1)j},p) \geq u\left(b,\frac\eps b\right)b$.
\end{lemma}

\begin{proof}
We prove this by induction on $j$. If $j=1$, the property holds vacuously. Suppose now that the property holds for $j$ and we want to prove it for $j+1$. If $i<m_{j+1}-1$, $z_{(i+1)(j+1)}=z_{(i+1)j}$ and $z_{i(j+1)}=z_{ij}$, so we simply apply the induction hypothesis. If $i=m_{j+1}-1$, then $z_{i(j+1)}=z_{ij}$ and $z_{(i+1)(j+1)} = \frac{z_{ij}+Sz_{m_jj}}2$. Using Lemma~\ref{l4}, we have that
$$d(Sz_{m_jj},p) \leq d(z_{m_jj},p) \leq d(z_{ij},p) \leq b.$$
In addition,
$$d(z_{ij},Sz_{m_jj}) \geq \frac{d(z_{ij},Sz_{m_jj})}2 = d(z_{ij},z_{(i+1)(j+1)}) = d(z_{i(j+1)},z_{(i+1)(j+1)}) \geq \eps = \frac\eps b \cdot b,$$
so, applying Proposition~\ref{u-prop},
$$d(z_{(i+1)(j+1)},p) \leq d(z_{i(j+1)},p) - u\left(b,\frac\eps b\right)b.$$
\end{proof}

We shall now construct a sequence $p_1 < p_2 < \ldots$ such that for any $k \geq 1$, $m_{p_k}=k$ (so $z_{kp_k}=y_{p_k}$) and for all $j \geq p_k+1$, $m_j \geq k+1$ and $z_{kj} = z_{kp_k}$, and $p_k$ is optimal in this regard, i.e. either $z_{k(p_k-1)}$ is not defined or $z_{k(p_k-1)} \neq z_{kp_k}$ (which makes it uniquely determined). We shall denote, for all $k \geq 1$, $x_k:=y_{p_k}$. We will also show that for all $k$, $d(x_k,x_{k+1}) \geq d\left(x_k, \frac{x_k+Sx_k}2\right)$, i.e. $2d(x_k,x_{k+1}) \geq d(x_k,Sx_k)$.

It is clear that one must have $p_1:=1$ (so $x_1=x$). Assume that we have constructed the sequence up to $p_k$ and we want to find the value of $p_{k+1}$.

We know that $m_{p_k+1} \geq k+1$, but since $m_{p_k+1} \leq m_{p_k}+1 = k+1$, $m_{p_k+1} = k+1 = m_{p_k}+1$. Thus the $(p_k+1)$th line was obtained using Construction case II, so
$$z_{(k+1)(p_k+1)} =\frac{x_k+Sx_k}2.$$

In the case where for all $t \geq p_k+1$, $z_{(k+1)t}=z_{(k+1)(p_k+1)}$, in order to simply put $p_{k+1}:=p_k+1$, we must also show that for all $j \geq p_k+2$, $m_j \geq k+2$. Assume that there is a $j \geq p_k+2$ with $m_j < k+2$, i.e. $m_j=k+1$. Since $m_{j-1} \geq k+1$, the $j$th sequence must have necessarily been obtained via Construction case I with $i=k$, so
\begin{equation}\label{e1}
d(z_{k(j-1)},z_{(k+1)(j-1)}) < d(z_{k(j-1)},z_{(k-1)(j-1)}),
\end{equation}
\begin{equation}\label{e2}
d\left(z_{k(j-1)},\frac{z_{k(j-1)} + Sz_{m_{j-1}(j-1)}}2\right) \geq d(z_{k(j-1)},z_{(k-1)(j-1)}),
\end{equation}
and $z_{(k+1)j}=\frac{z_{k(j-1)} + Sz_{m_{j-1}(j-1)}}2$. Since, by our assumption, $z_{(k+1)j} = z_{(k+1)(p_k+1)} = z_{(k+1)(j-1)}$, we have that \eqref{e2} yields
$$d(z_{k(j-1)},z_{(k+1)(j-1)}) \geq d(z_{k(j-1)},z_{(k-1)(j-1)}),$$
which contradicts \eqref{e1}. In this case $x_{k+1}=z_{(k+1)(p_k+1)} =\frac{x_k+Sx_k}2$, so $d(x_k,x_{k+1}) = d\left(x_k, \frac{x_k+Sx_k}2\right)$.

Assume now that there is a $t \geq p_k+2$ with $z_{(k+1)t}\neq z_{(k+1)(p_k+1)}$ and take it to be minimal ({\it a posteriori} it will be unique). Then, since $m_{t-1} \geq k+1$, we have that the $t$th sequence must have been obtained via Construction case I with $i=k$, so $m_t=k+1$,
\begin{equation}\label{estar}
d(z_{k(t-1)},z_{(k+1)(t-1)}) \geq d(z_{k(t-1)},z_{(k-1)(t-1)})
\end{equation}
and
\begin{equation}\label{e3}
d\left(z_{k(t-1)},\frac{z_{k(t-1)} + Sz_{m_{t-1}(t-1)}}2\right) \geq d(z_{k(t-1)},z_{(k-1)(t-1)}).
\end{equation}
For all $s \leq k$, $t-1 \geq p_k \geq p_s$, so $z_{st}=z_{s(t-1)}$, and since $z_{(k+1)t} = \frac{z_{k(t-1)} + Sz_{m_{t-1}(t-1)}}2$, \eqref{e3} yields
\begin{equation}\label{e4}
d(z_{kt},z_{(k+1)t}) \geq d(z_{kt},z_{(k-1)t}).
\end{equation}
We will now show that for all $j \geq t+1$, $m_j \geq k+2$ and $z_{(k+1)j}=z_{(k+1)t}$, so we may put $p_{k+1}:=t$. Start with $j:=t+1$. Suppose that $m_{t+1} < k+2$, i.e. $m_{t+1}=k+1$. Then the $(t+1)$th sequence must have been obtained via Construction case I with $i=k$, so $d(z_{kt},z_{(k+1)t}) < d(z_{kt},z_{(k-1)t})$, which contradicts \eqref{e4}. Since then $m_{t+1} \geq k+2 = m_t+1$, the $(t+1)$th sequence must have been obtained via Construction case II, so $z_{(k+1)(t+1)}=z_{(k+1)t}$. Assume now that the property to be proven holds for the $j$th sequence and we want to prove it for the next one. By the induction hypothesis, $z_{(k+1)j}=z_{(k+1)t}$. Suppose that $m_{j+1} < k+2$, i.e. $m_{j+1}=k+1$. Then the $(j+1)$th sequence must have been obtained via Construction case I with $i=k$, so
\begin{equation}\label{e5}
d(z_{kj},z_{(k+1)j}) < d(z_{kj},z_{(k-1)j}).
\end{equation}
We have that for all $s \leq k$, $j \geq p_k \geq p_s$, so $z_{sj}=z_{st}$, and since, as stated before, $z_{(k+1)j}=z_{(k+1)t}$, \eqref{e5} yields that $d(z_{kt},z_{(k+1)t}) < d(z_{kt},z_{(k-1)t})$, which contradicts \eqref{e4}. Since then $m_{j+1} \geq k+2$, the $(j+1)$th sequence must have been obtained via Construction case I with $i \geq k+1$ or via Construction case II, so $z_{(k+1)(j+1)}=z_{(k+1)j}=z_{(k+1)t}$. In addition, by the minimality of $t$, we get that
$$z_{(k+1)(t-1)}= z_{(k+1)(p_k+1)} = \frac{x_k+Sx_k}2,$$
so \eqref{estar} yields that
$$d\left(x_k, \frac{x_k+Sx_k}2\right) < d(x_k,x_{k-1}),$$
while \eqref{e4} means that
$$d(x_k,x_{k+1}) \geq  d(x_k,x_{k-1}),$$
so
$$d(x_k,x_{k+1}) > d\left(x_k, \frac{x_k+Sx_k}2\right).$$

We have now finished constructing the $p_k$'s.

\begin{lemma}\label{l9}
Let $p \in Fix(S)$ and $k \in \N$. Then for all $n \geq p_k$, $d(y_n,p) \leq d(x_k,p)$.
\end{lemma}

\begin{proof}
Let $n \geq p_k$, so $m_n \geq k$. By Lemma~\ref{l4}, $d(z_{m_nn},p) \leq d(z_{kn},p)$. On the other hand $z_{m_nn}=y_n$ and since $n \geq p_k$, $z_{kn}=z_{kp_k}=x_k$, so the conclusion follows.
\end{proof}

\begin{lemma}\label{l2}
Let $p \in Fix(S)$, $\eps>0$ and $i \in \N$. Let $u$ be such that the property described by Proposition~\ref{u-prop} holds. Assume that $d(x_i,x_{i+1}) \geq \eps$. Then $ d(x_i,p) - d(x_{i+1},p) \geq u\left(b,\frac\eps b\right)b$.
\end{lemma}

\begin{proof}
Set $j:=p_{i+1}$. We have that $m_j = m_{p_{i+1}} = i+1$, so $i \in [1,m_j-1]$. Note that $z_{(i+1)j}=z_{(i+1)p_{i+1}}=x_{i+1}$ and that, since $j \geq p_i$, $z_{ij} = z_{ip_i}=x_i$. Now apply Lemma~\ref{l3}.
\end{proof}

\begin{lemma}\label{l5}
We have that $\lim_{n \to \infty} d(x_n,x_{n+1}) = 0$.
\end{lemma}

\begin{proof}
Here is where we use that $Fix(S) \neq \emptyset$. Let $p \in Fix(S)$. Assume that our conclusion is false, i.e. there is an $\eps > 0$ such that for all $N$ there is an $n >N$ such that $d(x_n,x_{n+1}) > \eps$. Denote, for all $n$, $a_n:=d(x_n,x_{n+1})$ and $c_n:=d(x_n,p) - d(x_{n+1},p)$. Put $d_0:=1$ and for all $n\geq 0$, put $d_{n+1}>d_n$ such that $a_{d_{n+1}} > \eps$. In particular, we have that for all $i \geq 1$, $a_{d_i}>\eps$. Let $u$ be such that the property described by Proposition~\ref{u-prop} holds and set $k:=\left\lceil\frac{b+1}{u\left(b,\frac\eps b\right)b}\right\rceil$. Applying Lemma~\ref{l2}, we get that for all $i \geq 1$, $c_{d_i}\geq u\left(b,\frac\eps b\right)b$. We may thus write
$$b \geq d(x_1,p) \geq d(x_1,p) - d(x_{d_k+1},p) = \sum_{n=1}^{d_k} c_n \geq \sum_{i=1}^k c_{d_i} \geq k \cdot u\left(b,\frac\eps b\right)b \geq b+1,$$
a contradiction.
\end{proof}

\begin{lemma}\label{la}
Let $n$ be such that $x_n=x_{n+1}=x_{n+2}$. Then:
\begin{enumerate}[(i)]
\item $x_n=Sx_n$;
\item for all $j \geq p_{n+2}$ and all $q \in [n,m_j]$, $z_{qj}=x_n$;
\item for all $j \geq p_{n+2}$, $y_j=x_n$;
\item for all $q \geq n$, $x_q=x_n$.
\end{enumerate}
\end{lemma}

\begin{proof}
\begin{enumerate}[(i)]
\item Put $j:=p_{n+2}$. Then $j > p_{n+1}$, so $j-1 \geq p_{n+1} > p_n$. Then we get that $z_{(n+1)(j-1)}=z_{(n+1)p_{n+1}}=x_{n+1}$ and $z_{n(j-1)}=z_{np_{n}}=x_n$, so, since $x_n=x_{n+1}$, $d(z_{(n+1)(j-1)},z_{n(j-1)})=0$.

Assume that the $j$th sequence was obtained using Construction case I. Then we must have, since $m_j=n+2$, $d(z_{(n+1)(j-1)},z_{(n+2)(j-1)})< d(z_{(n+1)(j-1)},z_{n(j-1)})=0$, a contradiction. Thus, it was obtained using Construction case II, so
$$x_{n+2} = z_{(n+2)j} = \frac{z_{(n+1)(j-1)}+Sz_{(n+1)(j-1)}}2 = \frac{x_{n+1}+Sx_{n+1}}2.$$
However, $x_{n+1}=x_{n+2}$, so $0=d(x_{n+1},x_{n+2})=d(x_{n+1},Sx_{n+1})/2$. Thus, $x_{n+1}=Sx_{n+1}$, i.e. $x_n=Sx_n$.

\item It is clear that this holds for $j=p_{n+2}$. Assume that it holds for a $j \geq p_{n+2}$ -- since then $j \geq p_{n+1} +1$, $m_j \geq n+2$ -- and we want to prove it for $j+1$ -- since $j+1\geq p_{n+2}+1$, $m_{j+1}\geq p_{n+3}$. Thus, the $(j+1)$th sequence was obtained either using Construction case I with $i \geq n+2$ or using Construction case II. Let $q \in [n,m_j]$. Then either $z_{q(j+1)}=z_{qj}=x_n$ (by the induction hypothesis) or
$$z_{qj} = \frac{z_{(q-1)j}+Sz_{m_jj}}2,$$
but this can happen only if $q-1 \geq n+2$, so (by the induction hypothesis) $z_{(q-1)j}=x_n$ and $z_{m_jj}=x_n$. Since in addition we know that $x_n=Sx_n$, we have that
$$z_{qj} = \frac{x_n+Sx_n}2=x_n.$$

\item Let $j \geq p_{n+2}$ and put $q:=m_j$ in the above.

\item Let $q \geq n+3$. Then $p_q \geq p_{n+2}$ and $x_q=y_{p_q}=x_n$.
\end{enumerate}
\end{proof}

\begin{lemma}\label{l6t}
Let $n \geq 2$ be such that $d(x_n,x_{n+1}) < d(x_n,x_{n-1})$. Then for all $u \geq p_{n+1}$, $2d(x_n, x_{n-1}) \geq d(x_n, Sy_v)$.
\end{lemma}

\begin{proof}
Let $v \geq p_{n+1}$, so for all $i \in [n-1,n+1]$, $z_{iv}=x_i$. Assume that $2d(x_n, x_{n-1}) < d(x_n, Sy_v)$, so
$$d(z_{nv},z_{(n-1)v}) \leq \frac12 d(z_{nv}, Sz_{m_vv}) = d\left(z_{nv},\frac{z_{nv}+Sz_{m_vv}}2\right).$$
We also know that $d(z_{nv},z_{(n+1)v})<d(z_{nv},z_{(n-1)v})$, so the $(v+1)$th sequence is obtained by Construction case I with $i \leq n$, so $m_{v+1} \leq n+1$. On the other hand $v+1 \geq p_{n+1}+1$, so $m_{v+1} \geq n+2$, a contradiction.
\end{proof}

\begin{lemma}\label{l6}
Let $n \geq 2$ be such that $d(x_n,x_{n+1}) < d(x_n,x_{n-1})$. Then for all $q \geq n$, $2d(x_n, x_{n-1}) \geq d(x_n, x_q)$.
\end{lemma}

\begin{proof}
Clearly, the conclusion holds for $q=n$ and $q=n+1$. Let $q \geq n+1$ and assume that the conclusion holds for $q$. We want to prove that it also holds for $q+1$. Put $v:=p_{q+1}-1$. Since $v \geq p_q$, $z_{qv}=x_q$, so
$$x_{q+1}=\frac{z_{qv}+Sy_v}2 = \frac{x_q+Sy_v}2 .$$
By the induction hypothesis, we have that $2d(x_n, x_{n-1}) \geq d(x_n, x_q)$ and since $v \geq p_q \geq p_{n+1}$, by Lemma~\ref{l6t} we have that $2d(x_n, x_{n-1}) \geq d(x_n, Sy_v)$, so
$$d(x_n,x_{q+1}) \leq \frac12d(x_n, x_q) + \frac12d(x_n, Sy_v) \leq 2d(x_n, x_{n-1}).$$
\end{proof}

\begin{proposition}\label{l7}
The sequence $(x_n)$ is Cauchy.
\end{proposition}

\begin{proof}
Let $\eps>0$. We want an $M$ such that for all $n$, $m\geq M$, $d(x_n,x_m) \leq \eps$. Put $\delta:=\frac\eps4$. By Lemma~\ref{l5}, there is an $N$ such that for all $k \geq N$, $d(x_k,x_{k+1}) \leq \delta$.\\[2mm]

{\bf Case I.} We have that $x_N=x_{N+1}=x_{N+2}$.\\[1mm]

Then, by Lemma~\ref{la}, for all $q \geq N$, $x_q=x_N$, so we may take $M:=N$.\\[2mm]

{\bf Case II.} There is a $k \in \{N,N+1\}$ such that $x_k \neq x_{k+1}$.\\[1mm]

Put $\rho:=d(x_k,x_{k+1})>0$. Again, by Lemma~\ref{l5}, there is a $p \geq k+1$ such that $d(x_p,x_{p+1}) \leq \frac\rho2$, so there is an $M \in [k+1,p]$ such that $d(x_M, x_{M+1})<d(x_M,x_{M-1})$. Thus, by Lemma~\ref{l6}, for all $q \geq M$, $2d(x_M,x_{M-1}) \geq d(x_M,x_q)$.

On the other hand, since $M \geq k+1$, $M-1\geq k\geq N$, so $d(x_M,x_{M-1}) \leq \delta$ and for all $q \geq M$, $d(x_M,x_q)\leq 2\delta$.

Let $n$, $m \geq M$. Then $d(x_n,x_m)\leq d(x_M,x_n) + d(x_M,x_m) \leq 4\delta =\eps$.
\end{proof}

Since $X$ is complete, $(x_n)$ is convergent. We denote its limit by $p$.

\begin{lemma}\label{l6b}
Let $n \geq 2$ be such that $d(x_n,x_{n+1}) < d(x_n,x_{n-1})$. Then for all $q \geq n+1$, $2d(x_n, x_{n-1}) \geq d(x_n, Sx_q)$.
\end{lemma}

\begin{proof}
Let $q \geq n+1$ and put $v:=p_q \geq p_{n+1}$. Then $x_q=y_v$ and the conclusion follows by Lemma~\ref{l6t}.
\end{proof}

\begin{lemma}\label{l7t}
We have that $\lim_{n \to \infty} d(x_n,Sx_n) = 0$.
\end{lemma}

\begin{proof}
It follows immediately from Lemma~\ref{l5} and the fact that for all $k$, $2d(x_k,x_{k+1}) \geq d(x_k,Sx_k)$.
\end{proof}

Thus, by the establishing properties of $S$, $p \in Fix(S) = Fix(T)$.

\begin{proposition}\label{l10}
The sequence $(y_n)$ converges to $p$.
\end{proposition}

\begin{proof}
Let $\eps>0$. We want an $N$ such that for all $n \geq N$, $d(y_n,p) \leq \eps$. Let $k$ be such that $d(x_k,p) \leq \eps$ and put $N:=p_k$. Then the conclusion follows by Lemma~\ref{l9}.
\end{proof}

We may also show that $(y_n)$ is Cauchy without referring to its limit, by first proving the following analogue of Lemma~\ref{l6}.

\begin{lemma}\label{ly1}
Let $n \geq 2$ be such that $d(x_n,x_{n+1}) < d(x_n,x_{n-1})$. Then for all $v \geq p_{n+1}$, $2d(x_n, x_{n-1}) \geq d(x_n, y_v)$.
\end{lemma}

\begin{proof}
If $v=p_{n+1}$, $y_v=x_{n+1}$, so the conclusion holds by our hypothesis. Take an $v \geq p_{n+1}$ and assume the conclusion holds for all $l \in [p_{n+1},v]$. We want to prove it for $v+1$.  Since $v+1 \geq p_{n+1} +1$, $m_{v+1} \geq n+2$. Set $i:=m_{v+1}-1 \geq n+1$. Then
$$y_{v+1} = \frac{z_{iv} + Sy_v}2.$$
Set
$$s:=\min\{ k \geq 0 \mid m_{v-k} \leq i\}.$$
Since $m_{p_{n+1}} = n+1 \leq i$, we have that $s \leq v - p_{n+1}$. We also have that $m_{v-s} \leq i$ and $m_{v-s+1} >i$ (noting, for the edge case $s=0$, that $m_{v+1} > i$). Since $m_{v-s} \geq m_{v-s+1} -1 > i-1$, we have that $m_{v-s}=i$. By the minimality of $s$, for all $k \in [0,s)$, $m_{v-k}>i$, i.e. for all $k \in (0,s]$, $m_{v-s+k} > i$, so for all $k \in [0,s]$, $z_{i(v-s+k)} = z_{i(v-s)} = z_{m_{v-s}(v-s)} = y_{v-s}$. In particular, for $k:=s$, $z_{iv}=y_{v-s}$. Put $l:=v-s$. Then, since $s \in [0,v - p_{n+1}]$, $l \in [p_{n+1},v]$. By the induction hypothesis, we have that
$$2d(x_n, x_{n-1}) \geq d(x_n, y_l) = d(x_n, z_{iv}).$$
In addition, by Lemma~\ref{l6t}, we have that
$$2d(x_n, x_{n-1}) \geq d(x_n, Sy_v),$$
so
$$d(x_n,y_{v+1}) \leq \frac12d(x_n, z_{iv}) + \frac12d(x_n, Sy_v) \leq 2d(x_n, x_{n-1}).$$
\end{proof}

\begin{proposition}\label{ly2}
The sequence $(y_n)$ is Cauchy.
\end{proposition}

\begin{proof}
Let $\eps>0$. We want an $M$ such that for all $n$, $m\geq M$, $d(y_n,y_m) \leq \eps$. Put $\delta:=\frac\eps4$. By Lemma~\ref{l5}, there is an $N$ such that for all $k \geq N$, $d(x_k,x_{k+1}) \leq \delta$.\\[2mm]

{\bf Case I.} We have that $x_N=x_{N+1}=x_{N+2}$.\\[1mm]

Then, by Lemma~\ref{la}, for all $j \geq p_{N+2}$, $y_j=x_N$, so we may take $M:=p_{N+2}$.\\[2mm]

{\bf Case II.} There is a $k \in \{N,N+1\}$ such that $x_k \neq x_{k+1}$.\\[1mm]

Put $\rho:=d(x_k,x_{k+1})>0$. Again, by Lemma~\ref{l5}, there is a $p \geq k+1$ such that $d(x_p,x_{p+1}) \leq \frac\rho2$, so there is an $s \in [k+1,p]$ such that $d(x_s, x_{s+1})<d(x_s,x_{s-1})$. Thus, by Lemma~\ref{ly1}, for all $v \geq p_{s+1}$, $2d(x_s,x_{s-1}) \geq d(x_s,y_v)$.

On the other hand, since $s \geq k+1$, $s-1\geq k\geq N$, so $d(x_s,x_{s-1}) \leq \delta$ and for all $v \geq p_{s+1}$, $d(x_s,y_v)\leq 2\delta$.

Put $M:=p_{s+1}$. Let $n$, $m \geq M$. Then $d(y_n,y_m)\leq d(x_s,y_n) + d(x_s,y_m) \leq 4\delta =\eps$.
\end{proof}

Since $X$ is complete, $(y_n)$ is convergent and we again denote its limit by $p$. Since $(x_n)$ is a subsequence of $(y_n)$, $(x_n)$ also converges to $p$, and so, as before, $p \in Fix(S) = Fix(T)$. If we do not want to use this detour via the convergence of $(x_n)$, we must prove the following analogue of Lemma~\ref{l7t}.

\begin{lemma}\label{ly3}
We have that $\lim_{n \to \infty} d(y_n,Sy_n) = 0$.
\end{lemma}

\begin{proof}
Let $\eps>0$. We want an $M$ such that for all $n\geq M$, $d(y_n,Sy_n) \leq \eps$. Put $\delta:=\frac\eps4$. By Lemma~\ref{l5}, there is an $N$ such that for all $k \geq N$, $d(x_k,x_{k+1}) \leq \delta$.\\[2mm]

{\bf Case I.} We have that $x_N=x_{N+1}=x_{N+2}$.\\[1mm]

Then, by Lemma~\ref{la}, for all $j \geq p_{N+2}$, $y_j=x_N$ and $Sy_j=Sx_N=x_N=y_j$, so we may take $M:=p_{N+2}$.\\[2mm]

{\bf Case II.} There is a $k \in \{N,N+1\}$ such that $x_k \neq x_{k+1}$.\\[1mm]

Put $\rho:=d(x_k,x_{k+1})>0$. Again, by Lemma~\ref{l5}, there is a $p \geq k+1$ such that $d(x_p,x_{p+1}) \leq \frac\rho2$, so there is an $s \in [k+1,p]$ such that $d(x_s, x_{s+1})<d(x_s,x_{s-1})$. Thus, by Lemma~\ref{ly1}, for all $v \geq p_{s+1}$, $2d(x_s,x_{s-1}) \geq d(x_s,y_v)$. In addition, by Lemma~\ref{l6t}, for all $v \geq p_{s+1}$, $2d(x_s,x_{s-1}) \geq d(x_s,Sy_v)$.

On the other hand, since $s \geq k+1$, $s-1\geq k\geq N$, so $d(x_s,x_{s-1}) \leq \delta$ and for all $v \geq p_{s+1}$, $d(x_s,y_v) + d(x_s,Sy_v)\leq 4\delta$.

Put $M:=p_{s+1}$ and let $n\geq M$. Then $d(y_n,Sy_n)\leq d(x_s,y_v) + d(x_s,Sy_v) \leq 4\delta =\eps$.
\end{proof}

\section{Acknowledgements}

This work has been supported by the German Science Foundation (DFG Project KO 1737/6-1) and by a grant of the Romanian Ministry of Research, Innovation and Digitization, CNCS/CCCDI -- UEFISCDI, project number PN-III-P1-1.1-PD-2019-0396, within PNCDI III.

\end{document}